\documentclass[11pt]{amsart}

\usepackage[mathscr]{eucal}
\usepackage{amssymb, amsmath,array, amscd}
\usepackage{enumerate}
\usepackage{graphicx}
\usepackage{url}
\usepackage[colorlinks,plainpages,backref]{hyperref}

\newtheorem{thm}{Theorem}[section]
\newtheorem{defn}[thm]{Definition}
\newtheorem{lem}[thm]{Lemma}
\newtheorem{cor}[thm]{Corollary}

\theoremstyle{definition}
\newtheorem{rem}[thm]{Remark}

\begin{document}

\title{Linear Tropicalizations}

\author{Mustafa Hakan G\"unt\"urk\"un}
\address{Ged\.iz University \\ 35665, Izm\.ir \\ Turkey} 
\email{hakan.gunturkun@gediz.edu.tr}

\author{Al\.i Ula\c{s} \"Ozg\"ur K\.i\c{s}\.isel}
\address{Middle East Technical University \\06531 Ankara \\  Turkey}
\email{akisisel@metu.edu.tr }


\begin{abstract}
Let $X$ be a closed algebraic subset of $\mathbb{A}^{n}(K)$ where $K$ is an algebraically closed field complete with respect to a nontrivial non-Archimedean valuation. We show that there is a surjective continuous map from the Berkovich space of $X$ to an inverse limit of a certain family of embeddings of $X$ called linear tropicalizations of $X$. This map is injective on the subset of the Berkovich space $X^{an}$ which  contains all seminorms arising from closed points of $X$. We show that the map is a homeomorphism if $X$ is a non-singular algebraic curve. Some applications of these results to transversal intersections are given. In particular we prove that there exists a tropical line arrangement which is realizable by a complex line arrangement but not realizable by any real line arrangement.    
\end{abstract}

\maketitle

\section{Introduction}

Suppose that $K$ is an algebraically closed field complete with respect to a nontrivial non-Archimedean valuation $\nu:K\rightarrow \mathbb{R}\cup \{-\infty\}$. The tropicalization of a closed algebraic subset $X$ of $K^n$ is the image $\nu(X)$ where $\nu$ stands for the coordinate-wise valution map into $(\mathbb{R}\cup \{-\infty\})^n$. Studying tropicalizations of varieties proved to be useful for answering several questions about classical algebraic geometry. In order to see various aspects of this connection one can consult \cite{gub, ims, mac, mik}. 

The tropicalization construction is extrinsic which means that it depends on the particular embedding of $X$ in $K^{n}$ and not just on $X$ as an abstract variety. In \cite{payne}, Payne considers all embeddings of $X$ into affine spaces (and later into toric varieties) and forms an inverse system of the resulting tropicalizations in the category of topological spaces (also see \cite{fgp}). Then the inverse limit of all such tropicalizations can be regarded as an intrinsic tropicalization of $X$. It is proven in the same paper that this intrinsic tropicalization is homeomorphic to $X^{an}$, which is the analytification of the variety $X$ in the sense of Berkovich.

In this paper we consider a smaller subset of embeddings whose components are linear combinations of the coordinate functions of the original embedding. One advantage of working with such embeddings is that they preserve degree. The composition of a linear embedding with the valuation will be called a linear tropicalization. In Theorem \ref{thm:inverse} it is shown that there is a surjective map from $X^{an}$ to the inverse limit of all linear tropicalizations which is also injective on the subset of $X^{an}$ that contains all seminorms arising from the closed points of $X$. However, the map fails to be injective once $dim(X)\geq 2$. We prove in Theorem \ref{thm:curve} that if $X$ is a nonsingular algebraic curve then $\pi$ is a homeomorphism. We also give an example of a singular curve for which $\pi$ is not injective.

The last section of the paper is devoted to some applications of these results to transversal intersections. In particular, Theorem \ref{thm:transverse} states that if $X$ and $Y$ are two varieties in $\mathbb{A}^n$ intersecting at a set of $m$ reduced points, then there exists a linear tropicalization $Trop_i$ such that $Trop_{i}(X)\cap Trop_{i}(Y)=Trop_{i}(X\cap Y)$. Determining tropicalizations of a variety having this property has been investigated by several authors, see \cite{mor, op, or}. Remark \ref{rem:bezout} explains that the tropical Bezout's theorem for curves can be used to give a proof of the classical Bezout's theorem. Corollary \ref{cor:linearrangements} shows that any line arrangement in $\mathbb{CP}^2$ admits a linear tropicalization such that along with the linearity and incidence relations, the transversality of the intersections are preserved. In Corollary \ref{cor:realizability}, we show that there exists a tropical line arrangement that can be realized by a complex line arrangement but not by any real line arrangement.

\noindent \textbf{Acknowledgments.}
We would like to thank J\'er\^ome Poineau for providing the example in Remark \ref{rem:noninjective} and for clarifying a point in a previous version of this paper. We would also like to thank Alp Bassa and Eric Katz for their valuable comments.

\section{Linear Tropicalizations }

Assume that $K$ is an algebraically closed field complete with respect to a non-Archimedean valuation $\nu$. Let us equip $\mathbb{T}=\mathbb{R}\cup \{-\infty\}$ with the topology extending the standard one on $\mathbb{R}$ such that open neighborhoods of $-\infty$ are semi-infinite open intervals and extend $\nu$ to a map from $K$ to $\mathbb{T}$ by setting $\nu(0)=-\infty$. We will identify $\mathbb{A}^n$ with $K^n$ and by abuse of notation set $\nu:\mathbb{A}^n\rightarrow \mathbb{T}^n$ to be the map given by coordinate-wise valuation. Let $X$ be a closed algebraic subset of $\mathbb{A}^n$ and $z_1,\ldots, z_n$ coordinate functions on $\mathbb{A}^n$. 

Let $X^{an}$ denote the Berkovich analytic space of $X$. Then $X^{an}$, as a set, contains the multiplicative seminorms on the ring of regular functions $K[X]$ of $X$ compatible with $\nu$ (a multiplicative seminorm $[\,\,]_{x}$ on $K[X]$ is compatible with $\nu$ if $[a]_{x}=\exp(\nu(a))$ for any $a\in K$). For each closed point $x\in X$ one has a multiplicative seminorm $[f]_{x}=|f(x)|$. Denote the set of all such seminorms coming from closed points of $X$ by $X^{cl}$. One has $X^{cl}\subsetneq X^{an}$. Indeed one major reason for studying $X^{an}$ is that $X^{cl}$ is totally disconnected with respect to weak topology whereas $X^{an}$ is not and retains important attributes of $X$ \cite{br, berk, tem}.

Let $N$ be a positive integer. Pick an $N$-tuple $f_{1},\ldots,f_{N}$ of elements of $K[z_1,\ldots,z_{n}]$. Using this $N$-tuple one can define a morphism $i:X\rightarrow \mathbb{A}^N$ by setting $i(x)=(f_{1}(x),\ldots,f_{N}(x))$. For each such morphism, $\nu\circ i: X\rightarrow \mathbb{T}^N$ gives a tropicalization of $X$ denoted by $Trop_i(X)$. A single tropicalization evidently depends on the particular choice of the morphism, hence is extrinsic. In order to obtain an intrinsic tropicalization of $X$, one can consider the family of ``all'' possible morphisms as follows: 

First, suppose that $\varphi: \mathbb{A}^N\rightarrow \mathbb{A}^M$ is a morphism such that $\varphi$ restricts to a homomorphism of algebraic groups $(K^{*})^N\rightarrow (K^{*})^M$. Important examples are projections to any subset of coordinates of $\mathbb{A}^N$. Say $i_N: X\rightarrow \mathbb{A}^N$ and $i_M:X\rightarrow \mathbb{A}^M$ are two morphisms as above. We say that $\varphi$ is equivariant with respect to $(i_{N}, i_{M})$ if $\varphi\circ i_N =i_M$. It is straightforward to see that if $\varphi$ is equivariant then it induces a linear map $Trop(\varphi): \mathbb{T}^N\rightarrow \mathbb{T}^M$ restricting to $Trop_{i_N}(X)\rightarrow Trop_{i_M}(X)$. If $\varphi$ is a projection to a subset of coordinates of $\mathbb{A}^N$, then $Trop(\varphi)$ is the projection to the same subset of coordinates of $\mathbb{T}^N$.    

Now consider the family $\mathcal{F}$ of all $i$ as above such that $i:X\rightarrow \mathbb{A}^N$ is an isomorphism onto its image, namely the family of all embeddings among such morphisms. Given two embeddings $i_{N}:X\rightarrow \mathbb{A}^N$ and $i_{M}:X\rightarrow \mathbb{A}^M$, set $i_{N}\times i_{M}:X\rightarrow \mathbb{A}^{N+M}$ to be the map given by $(i_N\times i_M) (x)=(i_N(x), i_M(x))$. Then $i_N\times i_M$ also belongs to $\mathcal{F}$. Furthermore, if $\varphi_{1}: \mathbb{A}^{N+M}\rightarrow \mathbb{A}^N$ and $\varphi_{2}: \mathbb{A}^{N+M}\rightarrow \mathbb{A}^M$ are projections to the first $N$ and last $M$ coordinates respectively, then $\varphi_{1}$ is equivariant with respect to $(i_{N}\times i_{M}, i_N)$ and $\varphi_{2}$ is equivariant with respect to $(i_{N}\times i_{M}, i_M)$. 

Let us now consider the partial order defined on $\mathcal{F}$ such that $i$ dominates $j$ if there exists an equivariant morphism $\varphi$ with respect to $(i,j)$. Since any two embeddings in $\mathcal{F}$ are dominated by an element of $\mathcal{F}$, one can consider the inverse limit of tropicalizations of $X$ with respect to this partial order. Set 
\[  Trop(X)=\varprojlim_{i\in \mathcal{F}} Trop_{i}(X)  \]
where the inverse limit is in the category of topological spaces. Then $Trop(X)$ can be viewed as an intrinsic tropicalization of $X$. In \cite{payne}, Payne proved that $Trop(X)$ is homeomorphic to the Berkovich space $X^{an}$ of $X$. More specifically, given  $f_{1},\ldots,f_{N} \in K[z_1,\ldots,z_{n}]$, let $\pi_i: X^{an} \rightarrow \mathbb{T}^N$ be the continuous map given by $x \mapsto (log[f_1]_x,\ldots, log[f_N]_x)$. Set 
\[ \pi(x)=\varprojlim_{i\in \mathcal{F}} \pi_{i}(x).                        \]
Then, Payne's theorem asserts that $\pi:X^{an}\rightarrow Trop(X)$ is a homeomorphism. 

We would like to consider an inverse limit over a smaller subset of isomorphisms, namely the ones preserving degree. 

\begin{defn} Suppose that $f_1,\ldots, f_N$ are linear combinations of $z_1,\ldots,z_n$ and $1$ in $K[z_{1},\ldots,z_{n}]$. We say that the morphism $i:X \rightarrow \mathbb{A}^N$ given by $i(x)=(f_1(x),\ldots, f_N(x))$ is linear. 
\end{defn}

Let $\mathcal{G}$ be the set of all linear embeddings $i:X\rightarrow \mathbb{A}^N$ for an arbitrary value of $N$. We will say that $Trop_i(X)$ is a linear tropicalization of $X$. If $i_{N}$ and $i_{M}$ are two linear embeddings then it is clear that $i_{N}\times i_{M}$ is a linear embedding. Put a partial order on $\mathcal{G}$ such that $i$ dominates $j$ if there exists an equivariant morphism $\varphi$ with respect to $(i,j)$ which is a composition of a linear embedding and a projection. Then we can talk about the inverse limit of tropicalizations of $X$ over the family $\mathcal{G}$.

\section{Inverse Limit of All Linear Tropicalizations }

\begin{thm} \label{thm:inverse} Let $X$ be a closed algebraic subset of $\mathbb{A}^n$. \\
Let $\pi=\displaystyle{\varprojlim_{i\in \mathcal{G}}\pi_i}: X^{an}\rightarrow  \displaystyle{\varprojlim_{i\in \mathcal{G}} Trop_{i}(X)}$. Then

\noindent (i) $\pi$ is surjective. \\
(ii) $\pi$ restricted to $X^{cl}$ is injective.  
\end{thm}

\begin{proof} (i) Since the map $\displaystyle{\varprojlim_{i\in \mathcal{F}} \pi_{i}}$ is surjective by Theorem 1.1 in \cite{payne} and $\mathcal{G}$ is a sub-poset of $\mathcal{F}$, we deduce that $\displaystyle{\varprojlim_{i\in \mathcal{G}}\pi_i}$ is surjective. 

\noindent (ii) Say $[\,\,]_{x}$ and $[\,\,]_{y}$ are two points in $X^{cl}$ corresponding to $x,y\in X$. Select a linear polynomial $f$ such that $f(x)=0$ but $f(y)\neq 0$. Then for any embedding of the form $i(z)=(f(z),f_{1}(z),\ldots,f_{N}(z))$, the first coordinates of $\pi_{i}(x)$ and $\pi_{i}(y)$ are different. We can find an infinite tower of such embeddings by changing the value of $N$ and this shows that $\pi(x)\neq \pi(y)$.  
\end{proof}

\begin{rem} \label{rem:noninjective} 
The map $\pi$ is not injective on all of $X^{an}$ if $n\geq 2$. Here is a concrete example (given by J. Poineau) for $n=2$: Let $r\in (0,1]$. For a single variable polynomial $p(x)$, let $|p(x)|_{r}=\max_{i} |a_{i}|r^{i}$ where $\displaystyle{p(x)=\sum_{i}a_{i}x^{i}}$.
Define two points $[\,\,]_{1}$, $[\,\,]_{2}$ in $(\mathbb{A}^{2})^{an}$ by 
\[ [f(x,y)]_{1}=|f(x,x^{2})|_{r}, \qquad      [f(x,y)]_{2}=|f(x,x^{2}+x^{3})|_{r}  \]
respectively. Then $[\,\,]_{1}\neq [\,\,]_{2}$ since, for instance, they take different values on the polynomial $f(x,y)=x^{2}-y$.  Suppose now that $f(x,y)=ax+by+c$ is linear. Then 
\[  [f(x,y)]_{1}=\max(|a|r, |b|r^{2}, |c|), \qquad  [f(x,y)]_{2}=\max(|a|r, |b|r^{2}, |b|r^{3}, |c|). \]
Since $r\leq 1$, these two numbers are equal. This implies that $\pi([\,\,]_{1})=\pi([\,\,]_{2})$, hence the map $\pi$ is not injective on $(\mathbb{A}^{2})^{an}$.
\end{rem}

In contrast to the remark above, the map $\pi$ is a bijection, hence a homeomorphism, when $X$ is a nonsingular algebraic curve. We will prove this result in several steps.

\begin{lem} The map $\pi$ is injective if $X=\mathbb{A}^{1}$ or if $X=B(x,r)$ is an open ball of radius $r$ in $\mathbb{A}^{1}$. 
\end{lem} 

\begin{proof} Since $K$ is algebraically closed, every element of the ring of regular functions $K[x]$ of $X$ can be factorized into linear polynomials. Thus, if two seminorms agree on linear polynomials, they must agree on all of $K[x]$. This shows that $\pi$ is an injection when $X=\mathbb{A}^{1}$. 

If $X=B(x,r)$ is an open ball in $\mathbb{A}^{1}$, then recall the argument in \cite{br}, section 1.2, or for a more general discussion about polynomial approximations, \cite{poi}: By the Weierstrass Preparation Theorem, any power series convergent on $B(x,r)$ can be written in the form $f(x)=p(x)u(x)$ where $p$ is a polynomial and $u(x)$ an invertible power series. The invertibility of $u(x)$ implies that $[u(x)]_{y}=1$ for any $[\,]_{y}\in X^{an}$. The argument then proceeds as in the case of $X=\mathbb{A}^{1}$. 
\end{proof} 

\begin{lem} \label{lem:radius}
Suppose that $X$ is a nonsingular algebraic curve in $\mathbb{A}^n$. There exists $r>0$ such that for every $x\in X$, the restriction of the map $\pi$ to $(B(x,r))^{an}$ is injective, where $B(x,r)$ denotes the ball of radius $r$ with center $x$. 
\end{lem} 

\begin{proof} 
Let $S$ be the set of points of $X$ such that at least one of the partial derivatives $\partial f/\partial x_{i}$ is zero. Without loss of generality, we may assume that $S$ is a finite set, if necessary by moving everything with an appropriate  linear automorphism
of $\mathbb{A}^{n}$. Let $r>0$ be the minimum of the distances between any two distinct points in $S$. 

Since $X$ is nonsingular, for every $p\in X$ there exists $i$ such that $\partial f/\partial x_{i} (p)\neq 0$ . Given $p$, select $i$ such that the distance $r_{p}$ from $p$ to the nearest point $q$ with  $\partial f/\partial x_{i} (q)= 0$ is maximal. If $p\in S$ then it is clear that $r_{p}\geq r$ by definition. Then for any $p\in X$ the inequality $r_{p}\geq r$ follows from the non-Archimedean triangle inequality. Consequently, for every point $x\in X$ the projection from $B(x,r)$ to one of the coordinate axes is a linear isomorphism to an open ball in $\mathbb{A}^{1}$. The claim then follows from the previous lemma.  
\end{proof} 

The next natural step would be to show that the map $\pi$ is locally one-to-one when $X$ is a nonsingular algebraic curve. The previous lemma says that $X$ can be covered by balls of radii uniformly bounded from below such that on the analytification of each ball the map $\pi$ is injective. However, we cannot immediately deduce that $\pi$ is locally one-to-one since $X^{an}$ need not be covered by analytifications of these open balls. In fact, the points on the minimal skeleton in $X^{an}$  do not admit any open ball neighborhoods. Therefore, the naive argument must be refined. 

Recall from  \cite{abbr} or \cite{bpr} that a skeleton $\Sigma\subset X^{an}$ for $X$ is a finite metrized graph such that its complement in $X$ is a disjoint union of open balls. Also recall that $X^{an}$ locally has the structure of an $\mathbb{R}$-tree. 

\begin{lem}  \label{lem:local}
If $X$ is a nonsingular algebraic curve in $\mathbb{A}^n$  then the map $\pi: X^{an}\rightarrow \displaystyle{\varprojlim_{i\in \mathcal{G}} Trop_{i}(X)}$ is locally one-to-one when restricted to any geodesic segment in $X^{an}$. 
\end{lem} 

\begin{proof} 
Let $\Sigma$ be an arbitrary skeleton for $X$. Refine the skeleton $\Sigma$ by adding the points in the set $S$ of Lemma \ref{lem:radius} (which are all of type 1) to the  vertex set of $\Sigma$, this can be done by \cite{abbr}, Lemma 3.15. The complement of the skeleton $\Sigma$ in $X$ is a disjoint union of open balls. 

Suppose now that $p\in X$. If $p$ is not on $\Sigma$, then it belongs to one of the open balls $B$ mentioned above. Select a type 1 point $q\in B$. By the proof of Lemma \ref{lem:radius}, there exists an open ball $B(q,r)$ such that $\pi$ is injective on $B(q,r)$ and such that there exists a point in $S$ at a distance $r$ from $q$. Then, since $S$ is a subset of the vertex set of $\Sigma$ we deduce that $B\subset B(q,r)$ and consequently $\pi$ is injective on $B$. Hence, $\pi$ is locally one-to-one in a neighborhood of each such $p$ and in particular on any geodesic segment containing $p$. 

It remains to prove that if $p\in  \Sigma$ then $\pi$ is locally one-to-one on any geodesic segment $l$ containing $p$ in its interior. By Lemma 3.15 of \cite{abbr} again, we can modify $\Sigma$ and assume that $l\subset \Sigma$. Say $q$ is a closed point in $X$, not in $\Sigma$, that retracts to $p$. Select an open ball $B(q,r)$ as in Lemma \ref{lem:radius} such that a point of $S$ is at a distance $r$ from $q$ and such that $\pi$ is one-to-one on $B(q,r)$. But then $B(q,r)$ must contain an open subset $\hat{l}$ of $l$ containing $p$ since points in $S$ belong to the skeleton. Since $\hat{l}\subset B(q,r)$, the map $\pi$ is one-to-one on $\hat{l}$.  
\end{proof} 
It remains to show that the map $\pi$ which is locally one-to-one when restricted to geodesic segments is actually globally one-to-one. We prove this below by essentially showing that the fundamental group of the inverse limit of linear tropicalizations cannot be more complicated than the fundamental group of the Berkovich space. 

\begin{thm} \label{thm:curve}
  Let $X$ be a nonsingular algebraic curve in $\mathbb{A}^{n}$. Then the map $\pi=\displaystyle{\varprojlim_{i\in \mathcal{G}}\pi_i}: X^{an}\rightarrow  \displaystyle{\varprojlim_{i\in \mathcal{G}} Trop_{i}(X)}$ is both injective and surjective. 
\end{thm} 

\begin{proof} Contrary to the claim, suppose that $\pi$ is not injective. Say $p\neq q \in X^{an}$ but $\pi(p)=\pi(q)$. Since $X^{an}$ has a tree structure, there exists a unique path $C$ from $p$ to $q$. The image $\pi(C)$ must contain a non-contractible loop, else there would exist a point $s\in C$ such that $\pi$ restricted to $C$  would not be locally one-to-one near $s$, contradicting Lemma \ref{lem:local}. Without loss of generality, we can replace $C$ by an inverse image of such a loop and suitably modify $p$, $q$, thus we can assume that $p\neq q$, $\pi(p)=\pi(q)$ and $\pi$ is one-to-one on the unique path $C$ joining $p$ to $q$ except at the endpoints. 

If we reinterpret the picture in terms of inverse limits, the fact that the non-contractible loop $\pi(C)$ survives to $\displaystyle{\varprojlim_{i\in \mathcal{G}} Trop_{i}(X)}$ but it becomes contractible in $\displaystyle{\varprojlim_{i\in \mathcal{F}} Trop_{i}(X)}$ implies the following: There exists a path $P\subset X$ joining $x\neq y\in X$  and a linear tropicalization $Trop_{i}(P)$ of $P$ homeomorphic to a circle and injective except at the end points, such that for every $j\in \mathcal{G}$ dominating $i$, $Trop_{j}(P)$ contains a non-contractible loop. However, select a linear polynomial $f$ such that $f(x)\neq f(y)$ and construct the embedding $j$ by adjoining $f$ to $i$ as the last coordinate. Then $Trop_{j}(P)$ is  both linear and injective on the path $P$, a contradiction. 

\end{proof} 

\begin{rem} 
If $X$ is a singular algebraic curve then $\pi$ is not necessarily injective. For instance, let $X$ be the plane curve $(y-x^{2})(y+x^{2})=0$ in $\mathbb{A}^{2}$. In a similar manner to the example in Remark \ref{rem:noninjective}, let  $r\in (0,1]$ and  $|p(x)|_{r}=\max_{i} |a_{i}|r^{i}$ where $\displaystyle{p(x)=\sum_{i}a_{i}x^{i}}$. Define 
\[ [f(x,y)]_{1}=|f(x,x^{2})|_{r}, \qquad      [f(x,y)]_{2}=|f(x,-x^{2})|_{r}.  \]
Both of these seminorms, viewed as seminorms on $\mathbb{A}^{2}$ induce well-defined seminorms on $X$. They are distinct (take $f(x,y)=y-x^{2}$), yet take the same values on linear polynomials. Hence $\pi([\,\,]_{1})=\pi([\,\,]_{2})$ and the map is not injective. 

The question of characterizing the singular algebraic curves $X$ for which the map $\pi$ is injective remains unanswered. A natural guess is that $\pi$ is injective if and only if the tangent cone at each singular point of $X$ is reduced, in particular $\pi$ is injective if $X$ has only nodal singularities. 
\end{rem}

\section{Transversal Intersections}

\begin{lem} \label{lem:anint} Say $X,Y\subset \mathbb{A}^{n}$ are closed algebraic subsets intersecting transversally. Then 

(a) $(X\cap Y)^{an}=X^{an}\cap Y^{an}$. 

(b) $(X\cap Y)^{cl}=X^{cl}\cap Y^{cl}$. 

\end{lem} 

\begin{proof} Say $I(X)$ and $I(Y)$ are ideals of $X$ and $Y$ respectively, so that $K[X]=K[z_{1},\ldots,z_{n}]/I(X)$ and $K[Y]=K[z_{1},\ldots,z_{n}]/I(Y)$. Since $X$ and $Y$ are intersecting transversally, $rad(I(X)+I(Y))=I(X)+I(Y)$. Therefore $K[X\cap Y]=K[z_{1},\ldots,z_{n}]/(I(X)+I(Y))$. 

(a) Suppose that $[\,\,]_{x}\in (X\cap Y)^{an}$. Then $[\,\,]_{x}$ is a multiplicative seminorm on $K[z_{1},\ldots,z_{n}]$ vanishing on $I(X)+I(Y)$. Therefore $[\,\,]_{x}$ vanishes both on $I(X)$ and $I(Y)$. This implies that $[\,\,]_{x}$ is a multiplicative seminorm on $K[X]$ and $K[Y]$, hence $(X\cap Y)^{an}\subset X^{an}\cap Y^{an}$. 

Conversely, say $[\,\,]_{x}\in X^{an}\cap Y^{an}$. Then $[\,\,]_{x}$ vanishes both on $I(X)$ and $I(Y)$. But the kernel of a seminorm is an ideal, therefore $[\,\,]_{x}$ vanishes on $I(X)+I(Y)$. So $X^{an}\cap Y^{an}\subset (X\cap Y)^{an}$ and we get the desired equality. 

(b) This is clear. 

\end{proof}

If $X$ consists of a single reduced point, then $K[X]=K$, therefore $X^{an}$ consists of a single multiplicative seminorm, namely the one coming from the valuation on $K$. Similarly, if $X$ consists in $m$ distinct reduced points, then $X^{an}$ has $m$ elements and $X^{cl}=X^{an}$ in this case. 

\begin{thm} \label{thm:transverse}
Suppose that $X,Y\subset \mathbb{A}^{n}$ where $X$ and $Y$ intersect at a set of $m$ reduced points. Then there exists a linear tropicalization $Trop_{i}$ of $X\cup Y$ such that $Trop_{i}(X)\cap Trop_{i}(Y)$ consists of $m$ distinct points and $Trop_{i}(X)\cap Trop_{i}(Y)=Trop_{i}(X\cap Y)$. 
\end{thm}

\begin{proof} Since $X\cap Y$ consists of $m$ distinct reduced points, $(X\cap Y)^{an}=(X\cap Y)^{cl}$ is a set of $m$ points. By Theorem  \ref{thm:inverse} parts (a) and (b), $\pi=\displaystyle{\varprojlim_{i\in \mathcal{G}}\pi_i}: (X\cap Y)^{an}\rightarrow  \displaystyle{\varprojlim_{i\in \mathcal{G}} Trop_{i}(X\cap Y)}$ is both surjective and injective, therefore there exists a linear tropicalization $Trop_{i}$ of $X\cap Y$
that has exactly $m$ points. Note that for every $j$ dominating $i$ in the inverse system, $Trop_{j}(X\cap Y)$ will also have $m$ points. Indeed, suppose that $j$ dominates $i$ and $Trop_{j}(X\cap Y)$ contains  $m+1$ points $p_{1},\ldots,p_{m+1}$. Select a sequence $k_{l}\in \mathcal{G}$ such that $k_{1}=j$ and the maps $Trop_{k_{l+1}}(X\cap Y) \rightarrow Trop_{k_{l}}(X\cap Y)$ are projections. Preimages of $p_{i}$ and $p_{j}$ in each $Trop_{k_{l}}$ are distinct if $i\neq j$. Consequently, there exist distinct points $q_{1},\ldots,q_{m+1}\in (X\cap Y)^{an}$ mapping to $p_{1},\ldots,p_{m+1}$ respectively. This is a contradiction.  

The inclusion $Trop(X\cap Y)\subset Trop(X)\cap Trop(Y)$ holds regardless of which tropicalization is chosen. Conversely, we claim that for the tropicalization chosen above $Trop_{i}(X)\cap Trop_{i}(Y)\subset Trop_{i}(X\cap Y)$. Suppose not. Then there exist $m+1$ distinct points $p_{1},\ldots,p_{m+1}\in Trop_{i}(X)\cap Trop_{i}(Y)$. Select a sequence $k_{l}\in \mathcal{G}$ such that $k_{1}=i$ and the maps $Trop_{k_{l+1}}(X\cup Y)\rightarrow Trop_{k_{l}}(X\cup Y)$ are projections. Then, the restrictions of these maps to $Trop_{k_{l+1}}(X)$ and $Trop_{k_{l+1}}(Y)$ are also projections. As in the previous paragraph, the preimages of $p_{i}$ and $p_{j}$ on each $Trop_{k_{l}}(X)\cap Trop_{k_{l}}(Y)$ are distinct if $i\neq j$.  Therefore there exist $q_{1},\ldots,q_{m+1}\in X^{an}\cap Y^{an}$ mapping to $p_{1},\ldots,p_{m+1}$ respectively. But $X^{an}\cap Y^{an}=(X\cap Y)^{an}=(X\cap Y)^{cl}$ so the cardinality of $X^{an}\cap Y^{an}$ must be $m$. This contradiction finishes the proof. 
\end{proof}

\begin{rem} \label{rem:bezout}
Suppose that $X$ and $Y$ are plane curves intersecting transversally at $m$ points in $\mathbb{CP}^{2}$ (or in a projective plane over any algebraically closed field). We may visualize $X$ and $Y$ in $\mathbb{P}^2(K)$ where $K$ is the completion of the field of Puiseux series over $\mathbb{C}$. Applying Theorem \ref{thm:transverse} we see that there exists a linear tropicalization $Trop_{i}$ of the projective plane such that $Trop_{i}(X)$ and $Trop_{i}(Y)$ intersect at $m$ points. This can be used to produce a proof of (the classical) Bezout's theorem for plane curves using the tropical Bezout's theorem \cite{gath, rst}. 
\end{rem}

\begin{cor} \label{cor:linearrangements}
Let $X$ be a line arrangement in $\mathbb{CP}^2$ (or over any algebraically closed field). Then there exists a linear tropicalization $Trop_{i}$ of $X$ such that all tropical lines in this tropicalization intersect transversally. 
\end{cor}

\begin{proof} As in the previous remark, visualize the line arrangement in the projective plane over the completion of the field of Puiseux series. Then the result immediately follows from Theorem \ref{thm:transverse}.\end{proof} 

Recall that a tropical prevariety, namely a closed set with respect to the tropical semifield operations is said to be realizable (or a tropical variety) if it is the image $\nu$ of an algebraic variety over a non-Archimedean valued field $(K,\nu)$ \cite{rst}. 

\begin{cor} \label{cor:realizability}
There exists a  planar tropical line arrangement which is realizable by a complex line arrangement but not realizable by any real line arrangement. 
\end{cor}

\begin{proof} The Hessian configuration in $\mathbb{CP}^2$, namely the complex $(4,3)$-net, has a tropicalization $L$ that commutes with the intersection of the lines, by Corollary \ref{cor:linearrangements}. However, the abstract $(4,3)$-net does not have any real embeddings. Therefore $L$ is not realizable by any real line arrangement. (A detailed study of the Hessian configuration can be found in \cite{art} and k-nets in \cite{yuz}. The nonexistence of real 4-nets is proved in Lemma 2.4 in \cite{cord}.)
\end{proof}

\end{document}